\documentclass[12pt]{amsart}
\usepackage{amsmath}
\usepackage{amssymb}
\newtheorem{theorem}{Theorem}
\newtheorem{proposition}[theorem]{Proposition}
\newtheorem{lemma}[theorem]{Lemma}

\newenvironment{proof*}{\noindent {}}{\hfill
$\Box$ \vskip 2mm}
\newcommand{\B}{\mathbb B}
\newcommand{\C}{\mathbb C}
\newcommand{\D}{\mathbb D}

\newcommand{\N}{\mathbb N}

\newcommand{\Om}{\Omega}
\newcommand{\eps}{\varepsilon}
\newcommand{\vphi}{\varphi}

\def\Re{\operatorname{Re}}
\def\Im{\operatorname{Im}}

\title[Domains in $\C^n$ with locally Levi-flat boundaries]
{A characterization of domains in $\C^n$ with locally Levi-flat boundaries}

\author{Nikolai Nikolov, Pascal J. Thomas}

\address{Institute of Mathematics and Informatics\\ Bulgarian Academy
of Sciences\\ Acad. G. Bonchev 8, 1113 Sofia,
Bulgaria}\email{nik@math.bas.bg}

\address{Universit\'e de Toulouse\\ UPS, INSA, UT1, UTM \\
Institut de Math\'e\-matiques de Toulouse\\
F-31062 Toulouse, France} \email{pthomas@math.univ-toulouse.fr}

\subjclass[2010]{32F45, 32T27}

\keywords{Levi-flat boundary, Kobayashi metric, Bergman metric, Bergman kernel}

\begin{document}

\begin{thanks}{This paper was written during the stay if the
first-named author at the Paul Sabatier University, Toulouse in
October, 2011, and the stay of the second-named author at the
Institute if Mathematics and Informatics, BAS in November, 2011,
supported by the CNRS--BAS programme ``Convention d'\'echanges'',
No 23811.}
\end{thanks}

\begin{abstract} A domain in $\C^n$ with Levi-flat boundary near a given
point is characterized in terms of the boundary behavior of the
Kobaya\-shi or Bergman metrics, or of the Bergman kernel. Some
results are given in the case of intermediate values of the
rank of the Levi form.
\end{abstract}

\maketitle

\section{The main results}

This note is motivated by Ohsawa's question \cite[Q2]{Ohs} about
the characterization of a bounded pseudoconvex domain
$\Om\subset\C^n$ with locally Levi-flat boundary (i.e. the rank of
the Levi form is zero) in terms of the boundary behavior of the
Bergman metric $\beta_\Om.$  This follows from some of the results
of Siqi Fu's Ph. D. dissertation \cite{Fu1}, which have been made
more widely available recently \cite{Fu2}. 
We shall however give an answer to
Ohsawa's question,
as well as to a similar question for the Kobayashi
metric $\kappa_\Om$, using  weaker regularity assumptions
and somewhat different methods.  The 
Levi-flat case is only a special case of Fu's results, 
who also related the rank of the Levi form and rate of growth of the 
Bergman kernel. We provide similar results using our own methods,
 under smoothness hypotheses which are slightly different from Fu's.

Recall that
$$
\kappa_\Om(z;X)=\inf\{|t|:\exists\vphi\in\mathcal O(\Bbb D,\Om):
\vphi(0)=z,t\vphi'(0)=X\}$$
($\Bbb D$ is the unit disc) and
$$\beta_\Om(z;X)=m_\Om(z;X)/k_\Om(z),
$$
where $$k_\Om(z)=\sup\{|f(z)|:f\in L^2(\Om)\cap\mathcal
O(\Om),\;||f||_{L^2(\Om)}\le1\}$$ is the square root of Bergman
kernel (restricted to the diagonal) and
$$m_\Om(z;X)=\sup\{|f'(z)X|:f\in L^2(\Om)\cap\mathcal O(\Om),\;||f||_{L^2(\Om)}=1,\;f(z)=0\}.$$

Suppose that $p$ is a $\mathcal C^2$-smooth boundary point of $\Om.$  Then
for each $z\in D$ near $p$ there exists a unique point $p(z)\in
\partial \Om$ such that $z = p(z) + \delta_\Om(z) n_{p(z)}$, where
$\delta_\Om$ is the distance $\partial D$ and $n_{p(z)}$ is the
inner normal vector at $p(z).$ For such a $z$ and $X \in \C^n$,
there is a unique orthogonal decomposition $X=X_N+X_T$ where
$X_T\in T^{\C}_{p(z)}\partial \Om.$

\begin{theorem}
\label{koba} Let $p$ be a boundary point of a bounded domain
$\Om\subset\C^n$ such that $\partial \Omega$ is $\mathcal
C^2$-smooth near  $p$.

Then $\partial\Om$ is Levi-flat near $p$ if and only if there exist a
neighborhood $U$ of $p$ and a constant $c>1$ such that for any
$z\in\Om\cap U$ and any $X\in\C^n\setminus \{0\}$,
\begin{equation}
\label{growth} c^{-1}<\frac{\kappa_\Om
(z;X)}{\frac{\|X_N\|}{\delta_\Om(z)}+ \|X\|}<c.
\end{equation}
\end{theorem}

\begin{theorem}
\label{bergman} Let $p$ be a boundary point of a bounded
pseudoconvex domain $\Om\subset\C^n$ such that $\partial \Omega$
is $\mathcal C^2$-smooth near $p$. Then the following three
conditions are equivalent:

(i) $\partial\Om$ is Levi-flat near $p;$

(ii) there exist a neighborhood $U$ of $p$ and a constant $c>1$
such that for any $z\in\Om\cap U,$
\begin{equation}
\label{kernel}
c^{-1}<k_\Om(z)\delta_\Om(z)<c ;
\end{equation}

(iii) there exist a neighborhood $U$ of $p$ and a constant $c>1$
such that for any $z\in\Om\cap U$ and any $X\in\C^n\setminus
\{0\}$,
\begin{equation}
\label{metric}
c^{-1}<\frac{\beta_\Om
(z;X)}{\frac{\|X_N\|}{\delta_\Om(z)}+ \|X\|}<c.
\end{equation}
\end{theorem}

The proofs are given in sections \ref{sufkoba}, \ref{sufberg} and \ref{pfest}.

Section \ref{ranklevi} is devoted to the question of how we can recover
the rank of the Levi form at a boundary point from the growth of the Bergman
kernel near that point. There is a good fit in the $\mathcal C^\infty$-smooth case,
see Theorem \ref{smoothrank}.

In the last section \ref{sharp}, refinements of the above estimates are proved in the convex
and the planar cases.
\smallskip

{\noindent\bf Acknowledgement.} We thank Takeo Ohsawa and Peter Pflug
for reading an earlier version of this work.
\smallskip

{\noindent\bf Remark.} 
We had not realized, as we should have, that eighteen years ago,
 chapter IV of Fu's dissertation \cite{Fu1} 
had answered Ohsawa's question and various extensions to 
intermediate ranks and other metrics, in the smooth, pseudoconvex case. 
Although many technical tools are the same, such as choosing a normal form of the
coordinates to prove that certain polydisks (or more general sets) are contained
in the domain $\Omega$,
some differences should be noted between his work and ours.
We refer to the generally available version \cite{Fu2}.

\begin{itemize}
\item
Our Theorems \ref{koba} and \ref{bergman}
and Proposition \ref{rankcount} only require $\mathcal C^2$ smoothness, exploiting 
the optimal hypotheses of \cite{Krt}.

In \cite[Theorem 1.1]{Fu2}, which relates the rank of the Levi form
(assumed to be constant in a neighborhood of the base point) and the behavior of pseudometrics,
 $\Omega$ is assumed to be 
``smooth" near the point under consideration, although after examination it seems 
that the crucial tool \cite[Theorem 6.1]{Cat2} does not require
more than $\mathcal C^2$ smoothness and that
a $\mathcal C^3$ assumption is be enough to obtain \cite[Proposition 3.2]{Fu2}
(and perhaps one could improve that proof to require
only the $\mathcal C^2$ assumption). 
\item
In \cite[Theorem 1.1]{Fu2}, $\Omega$ is assumed to be pseudoconvex (because 
global plurisubharmonic functions are constructed, which we dispense with,
and \cite{Fu2} deals with the Sibony metric, which we do not treat); our Theorem \ref{koba} about the Kobayashi infinitesimal pseudometric does not require pseudoconvexity.

On the other hand,  \cite[Theorem 1.1]{Fu2} provides 
sharp estimates for pseudometrics applied to
tangent vectors in any direction, in terms of the Levi form.
\item
Our results about the relation between the local rank of the Levi form and
the growth of the Bergman kernel (Proposition \ref{rankcount}, Theorem \ref{smoothrank})
require $\mathcal C^\infty$ smoothness in one direction, and with our
method there is no way
to bound  the degree of smoothness required, so 
Fu's method, requiring implicitly only $\mathcal C^3$
smoothness, yields a stronger result there. 
\end{itemize}

\section{Proof of the sufficiency in Theorem \ref{koba}}
\label{sufkoba}

Let $\rho$ be a $\mathcal C^2$-smooth defining function of $\Om$ near $p.$ Suppose that there is some point $q\in\partial D$
near $p$ such that the Levi form of $\rho$ at $q$ (restricted on $T^{\C}_q\partial \Om$)
has a non-zero eigenvalue.

If it is a negative eigenvalue, it follows by \cite[Theorem 1.1]{Kra} that
$$\limsup_{x\to 0+}\delta_\Om^{3/4}(q_x)\kappa_\Om (q_x;n_q)<\infty,$$ where
$q_x=q+xn_q.$ Therefore the left hand-side estimate
in \eqref{growth} cannot hold for normal vectors.

On the other hand, if  the Levi form admits a positive eigenvalue,
Proposition \ref{tang} shows that the right-hand side of \eqref{growth}
cannot hold for all tangential vectors.

\begin{proposition}\label{tang}
Suppose that the Levi form admits a positive eigenvalue at $q\in\partial D$.
Then there is an $X\in T^{\C}_q\partial \Om$ such that
$$\limsup_{x\to 0+} \kappa (q_x;X) = \infty.
$$
\end{proposition}

In the case where all the eigenvalues are positive, stronger growth estimates
are known, see e.g. \cite[Chapter III, Theorem 3.1.1]{Fu1}.

\begin{proof}
Using a translation, a rotation, and the implicit function theorem, we may
always assume that $q=0$, and that $\Om=\{\rho<0\}$,
where near $0$,
$$\rho(z) = \Re z_1 + O((\Im z_1)^2+|z'|^2),
\mbox{ where } z'=(z_2, \dots,z_n).
$$
A further rotation lets us assume that the positive eigenvalue is in the $z_2$-direction,
and a dilation that it is equal to $1$.
Therefore
$$\rho(z) = \Re z_1 + |z_2|^2 + O((\Im z_1)^2+ |z_2| |z'| + |z'|^2),
\mbox{ where } z'=(z_3, \dots,z_n).
$$
Then
$$
\rho(z) \ge \Re z_1 + \frac12|z_2|^2 -C ((\Im z_1)^2+ |z'|^2),
$$
and since passing to a smaller defining function, thus to a larger domain,
can only decrease the Kobayashi metric, we may assume that $\rho$ has
this expression. At the cost of further dilations in $z_1, z_2$ and $z'$, we finally
reduce ourselves to
$$
\rho(z) = \Re z_1 + |z_2|^2 - ((\Im z_1)^2+ |z'|^2).
$$
We estimate $\kappa_\Om (z_\delta;X)$
where $z_\delta =(-\delta,0,\dots,0)$ ($\delta>0$ small enough) and $X=(0,1,0,\dots,0)$.
Let $\varphi$ be a holomorphic map from $\D$ to $\Om$ such that
$\varphi(0)=z_\delta$ and $\varphi'(0)=X$.
We will use with no further mention the fact that
the Taylor coefficients of $\varphi$ are bounded since $\Om$ is.
We have
$$
\varphi(\zeta) = (-\delta,\lambda \zeta,0,\dots,0) + \zeta^2 \psi(\zeta),
$$
where $\psi(\zeta)=(\psi_1(\zeta), \dots,\psi_n(\zeta))$ is bounded
for $|\zeta|\le \frac12$ (say).  Denote $\tilde \psi (\zeta)= \frac1\zeta
\left( \psi_1 (\zeta)-\psi_1(0) \right)$, which is also bounded.
Then
\begin{multline*}
\rho \circ \varphi(\zeta) = -\delta +\Re (\psi_1(0) \zeta^2)
+ \Re (\tilde \psi_1 (\zeta) \zeta^3)
\\
+|\lambda|^2 |\zeta|^2
\left| 1+\zeta \psi_2(\zeta)\right|^2 - (\Im (\psi_1(\zeta)  \zeta^2))^2
-|\zeta|^4 \sum_{j=3}^n |\psi_j(\zeta)|^2.
\end{multline*}

Choose $\zeta = re^{i\theta}$ with $\Re (\psi_1(0)e^{2i\theta} )\ge 0$.
Then since $\varphi(\D)\subset \Om$, we have for $0<r<1$,
$$
0> -\delta +|\lambda|^2 r^2 -C_1  r^3 -C_2  r^4
\ge -\delta +|\lambda|^2 r^2 -C_3  r^3.
$$
Choose $r=\delta^{1/3}$, we find $|\lambda| \le (1+C_3)^{1/2} \delta^{1/6}$, which means
that $\kappa_\Om (z_\delta;X) \gtrsim \delta^{-1/6}$,
and therefore goes to infinity as $\delta$ goes to $0$.
\end{proof}

\section{Proof of the sufficiency in Theorem \ref{bergman}}
\label{sufberg}

We prove in this section that each of the conditions \eqref{kernel} or \eqref{metric} implies
that $\partial\Omega$ admits a Levi flat portion in a neighborhood of $p$ by
proving that if $\partial\Omega$ is not Levi flat in any neighborhood of $p$,
then those estimates must fail.

Suppose  $\partial\Omega$ is not Levi flat in $U$. Recall that the
Levi form $\mathcal L\rho (q)$ is a semidefinite positive
Hermitian form on $T_q^{\C}\partial \Omega \simeq \C^{n-1}$, and
that if its rank $l_\Omega(q)$ is equal to $l$ and
$\C^{n-1}=T_1\oplus T_2$ with $\dim T_1=l$, $\mathcal L\rho
(q)|_{T_1}$ is definite positive, and $\mathcal L\rho
(q)|_{T_2}\equiv 0$.

Let now $k=\max_{q \in U\cap
\partial\Omega}l_{\Omega}(q).$ Since the rank is a lower semicontinuous function,
there exists a non empty open set $V_1$ such that $k=l_\Omega(q)$
for any $q \in V_1$. We choose such a $q$, and take coordinates so
that $q$ becomes the origin and $\rho(z)= \Re z_1 + \rho_2(z)$
with $\rho_2(z)=O(\|z\|^2)$, thus
$T_q^{\C}\partial\Omega=\{0\}\times \C^{n-1}$.

Furthermore we choose coordinates on $\C^{n-1}$ such that
$\mathcal L\rho (q)|_{\{0\}\times \C^{n-1-k}}$ $\equiv 0$ and
$\mathcal L\rho (q)|_{\C^{k}\times\{0\}}$ is definite positive.
This latter property is stable, more precisely there is a ball
about the origin $V_2\subset V_1$ such that
$(\partial\Omega)\cap(\C^{k+1}\times\{0\})\cap V_2$ is a strict
pseudoconvex boundary in $\C^{k+1}$. Let $\Omega'= \{ z' \in
\C^{k+1}: (z',0) \in \Omega\}$ (we will use the ${}'$ notation
freely to denote the first $k+1$ coordinates in what follows). It
is a pseudoconvex domain, a smoothing of $V_2 \cap \Omega'$ will
be strictly pseudoconvex, and there exists $V_3 \subset V_2$ such
that for $z=(z',0)\in (\Omega'\times \{0\}) \cap V_3$, then
$\delta_{\Omega'}(z') \asymp \delta_{\Omega}(z)$.

We now recall briefly how the Bergman kernel is estimated in
strictly pseudoconvex domains. Given $z'\in \Omega'$, there exists
a polydisk $P'_{z'} \Subset \Omega'$ with radii
$c\delta_{\Omega'}(z') $ in the complex normal direction and
$c\delta_{\Omega'}^{1/2}(z')$ in the complex tangential
directions. It follows by \cite[Theorem 3.5.1]{Hor} (see also
\cite{Die}) that
\begin{equation}
\label{spcest} k_{\Omega'}(z') \asymp
\delta_{\Omega'}(z')^{-1-k/2} \asymp \lambda_{2k+2}
(P'_{z'})^{-1/2},
\end{equation}
where $\lambda_m$ stands for the Lebesgue measure in real dimension $m$.

\begin{lemma}
\label{kernslice} There exist concentric balls about $q$,
$V_5\Subset V_4$ such that for any $z'\in V'_5$,
$$
k_{\Omega' \cap V'_4} (z')\asymp k_{\Omega \cap V_4} (z',0).
$$
By the localization property of the Bergman kernel
\cite[Proposition 1]{DFH},
$$
k_{\Omega'}(z')\asymp k_\Omega(z',0).
$$

\end{lemma}

Assuming Lemma \ref{kernslice}, we prove that \eqref{kernel} must
fail. Given any neighborhood $U$ of $p$, there is $q\in U$ to
which we can apply the Lemma and for $z=q+xn_q$, $x>0$,
\eqref{spcest} implies that
\begin{equation}
\label{berggrowth}
\delta_\Omega (z)k_\Omega (z) \asymp x^{-k/2}\to \infty
\mbox{ as }x\to0.
\end{equation}

We turn to the failure of \eqref{metric}. For a given $p$ and any neighborhood $U$,
we choose $q\in U$ as at the beginning of this section. Then there exists
a vector $X \in T_q^{\C}\partial\Omega$ such that
$$
\liminf_{x\to 0^+} x^{1/2} \beta_{\Omega} (q+xn_q) >0.
$$
Indeed, with the coordinates chosen above, let $X=(X',0)$ where
$\|X'\|=1$ and $X' \in \{0\}\times \C^k$.  Since $X'$ is a
complex-tangential vector for the strictly pseudoconvex domain
$\Omega \cap V_4$, it is known that $\beta_{\Omega'} (z',X')
\asymp \delta_{\Omega'}(z')^{-1/2}\|X'\|$ (see \cite{Die}).

The Ohsawa--Takegoshi extension theorem \cite{OT} applied to the
linear subspace $\C^{k+1}$ and the domain $\Omega$, implies that
$m_{\Omega'}(z',X') \preceq m_{\Omega}((z',0),\\(X',0))$.  In our
new coordinates, $z=q+xn_q=(-x,0,\dots,0)$, $\delta_\Omega(z)=
\delta_{\Omega'}(z')=x$. Then
\begin{multline*}
\beta_\Omega (z,X) = \frac{m_\Omega (z,X) }{k_\Omega (z) }
\succeq  \frac{m_{\Omega'} (z',X') }{k_{\Omega'} (z') } \\
= \beta_{\Omega'} (z',X') \asymp \delta_{\Omega'}(z')^{-1/2} = \delta_\Omega(z)^{-1/2}.
\end{multline*}

\begin{proof*}{\it Proof of Lemma \ref{kernslice}.}
The inequality $k_{\Omega' \cap V'_4} (z')\preceq k_{\Omega \cap
V_4} (z',0)$ follows easily from the Ohsawa--Takegoshi extension
theorem, as above.

To prove the converse inequality, we first invoke a result of
Sommer \cite{Som} and Kraut \cite{Krt} that yields a foliation of
a full neighborhood of $q$ by a $2k+2$-parameter family of complex
manifolds $F_\theta$ of complex dimension $n-1-k$ such that
$\rho|_{F_\theta}$ is constant for each $\theta$. Some comments
are in order : Sommer proved the theorem only in the case where
$\rho$ is $\mathcal C^4$-smooth, Kraut gave a new proof that is valid if
$\rho$ is $\mathcal C^2$-smooth. Both of them state the result only as a
foliation of the hypersurface $\{\rho=0\}$, but prove it for a
whole neighborhood (and in fact need this in order to carry out
their proofs). See \cite[p. 310]{Krt}: ``Da diese
charakteristischen Mannigfaltigkeiten notwending auf den Fl\"achen
$f=$ const. verlaufen, bl\"attern sie also diese
komplex-analytisch."

We need to have holomorphic parametrizations of the leaves depending continuously on $\theta$.
The leaves are obtained as integral manifolds of an integrable distribution
of vector fields, whose value at each $q\in U$ are vectors in the kernel of
$\mathcal L\rho (q)$.  Since $\rho$ is $\mathcal C^2$-smooth, those
vectors can be chosen to depend continuously on $q$.  The exponential maps
of those vector fields are analytic (because the vector fields are holomorphic
vector fields), and so not only their values, but also their derivatives with
respect to the parameters of each leaf, depend continuously on $\theta$.

Denote by $V_6$ a ball about $q$ small enough so that the kernel of $\mathcal L\rho(\zeta)$
remains transverse to $\C^{k+1}\times\{0\}$ with a uniformly bounded
angle for $\zeta \in V_6$, and therefore
so does $F_\theta \cap V_6$ for each relevant $\theta$. We then may parametrize
those manifolds by $F_\theta \cap \C^{k+1}\times\{0\} = \{(\theta,0)\}$,
provided $V_6$ is chosen small enough.

Since the leaves depend continuously on $\theta$, we can chose holomorphic maps
$\Phi_\theta : G_\theta \longrightarrow F_\theta \subset \C^n,$
where $G_\theta$
is a neighborhood of the origin in $\C^{n-1-k}$ such that the unit ball $\B^{n-1-k}
\subset G_\theta$, for any $\theta \in V'_4$, where $V_4 \subset V_6$
is a ball about the origin; and
furthermore, reducing $V_4$ as needed,
for $\theta \in V_4$, $\Phi_\theta (G_\theta) \subset V_6$ and
the Gram determinant of the image of
the standard basis of $\C^{n-1-k}$ by $D \Phi_\theta$ is bounded above and below
uniformly in $\theta \in V'_4$. For some ball $V_5 \Subset V_4$, we may also assume
that $P'_{z'} \subset V_4$ for any $z'\in V'_5$.

For any $f\in L^2(\Om\cap V_4)\cap\mathcal O(\Om\cap V_4)$, $z'\in V'_4$, then
$f\circ \Phi_{z'}$ is holomorphic in a neighborhood of $\bar {\B}^{n-1-k}$, so
$$
|f(z')|^2 = |f\circ \Phi_{z'} (0)|^2 \le
c_1 \int_{\B^{n-1-k}} |f\circ \Phi_{z'} (\zeta)|^2 \, d\lambda_{2(n-1-k)}(\zeta),
$$
where $c_1$ depends only on the dimension. For $z'\in V'_5$,
\begin{multline*}
\int_{P'_{z'}} |f(\xi)|^2 \,  d\lambda_{2(k+1)}(\xi)
\\
\le c_1 \int_{P'_{z'}} \int_{\B^{n-1-k}} |f\circ \Phi_\xi (\zeta)|^2 \,
d\lambda_{2(n-1-k)}(\zeta) \,  d\lambda_{2(k+1)}(\xi)
\\
\le
c_2 \int_{\tilde P_{(z',0)}} |f(\zeta)|^2 d\lambda_{2n}(\zeta)
\le c_2 \int_{\Om \cap V_4} |f(\zeta)|^2 d\lambda_{2n}(\zeta),
\end{multline*}
where $\tilde P_{(z',0)} = \bigcup \{  \Phi_{\xi} ( \B^{n-1-k}) : \xi \in P'_{z'} \}$,
and the requirements about transversality and the differential of $\Phi_\xi$
ensure that the Jacobian determinant involved in the change of variables remains
bounded.

Now if we are given a function $f$ in the unit ball of $L^2(\Om\cap V_4)\cap\mathcal O(\Om\cap V_4)$,
applying the mean value inequality as in \eqref{spcest} implies that
$|f(z',0)| \preceq \lambda_{2k+2} (P'_{z'})^{-1/2} \asymp k_{\Om'\cap V'_4}(z').$
\end{proof*}

\section{The rank of the Levi form}
\label{ranklevi}

Observe that \eqref{berggrowth} proves a bit more than the failure of
\eqref{kernel} when the boundary is not Levi flat. In particular, it still holds
when $k=0$ (Levi flat case), and gives an estimate of the growth of
the Bergman kernel in terms of the rank of the Levi form. Thus we have the
following corollary of the proof in section \ref{sufberg}.

\begin{proposition}
\label{rankcount}
Let $p$ be a boundary point of a bounded
pseudoconvex domain $\Om\subset\C^n$ such that $\partial \Omega$
is $\mathcal C^2$-smooth near $p$.

If there exist a neighborhood $U$ of $p$ and  constants $c>1$, $m>0$
such that for any $z\in\Om\cap U,$
\begin{equation}\label{est}
c^{-1}<k_\Om(z)\delta_\Om^m(z)<c,
\end{equation}
then $2(m-1) \in \N\cup\{0\}$ and there exists a non empty open
set $U_1 \subset U$ such that the Levi form of $\partial \Omega$
has constant rank $2(m-1)$ in $U_1$.

Conversely, if there exists a neighborhood $U$ of $p$ such that
the Levi form of $\partial\Omega$ has constant rank $2(m-1)$ in
$U$, then there is a neighborhood $U_1\subset U$ of $p$ such that
such that \eqref{est} holds for any $z\in\Om\cap U_1.$

As a consequence, denoting by $l_{\Omega}(p)$ the rank
of the Levi form of $\partial \Omega$ at $p$,
$$
\limsup_{\partial\Omega\ni q\to p} l_{\Omega}(q) =
 2 \limsup_{\Omega\ni z\to p} \frac{\log k_\Om(z)}{\log
1/\delta_\Om(z)} -1.
$$
\end{proposition}

More generally, one may conjecture that \eqref{est} implies that
the Levi form of $\partial\Omega$ has constant rank $2(m-1)$ near
$p.$ It is not difficult to see this (by dilatation of the
coordinates) if the rank is maximal (i.e. $n-1$). On the other
hand, \eqref{berggrowth} implies the conjecture when the rank is
minimal (i.e. $0$).

In general, it is difficult to say what happens to the foliation in
complex manifolds near a degeneracy point, where the rank of the Levi form
verifies
$l_\Omega(p)< \limsup_{q\to p, q\neq p} l_\Omega(q)$.  However, in the smooth case,
we may confirm the above conjecture with the aid of the Catlin multitype \cite{Cat}.

\begin{theorem}
\label{smoothrank}
Let $p$ be a boundary point of a bounded
pseudoconvex domain $\Om\subset\C^n$ such that $\partial \Omega$
is $\mathcal C^\infty$-smooth near $p$.
Then if \eqref{est} holds in a neighborhood of $p$,
we have $l_\Omega (p) =2(m-1)$.
\end{theorem}

\begin{proof}
In view of Proposition \ref{rankcount}, it is enough to consider the case
where the rank is not locally constant at $p$.  By lower semi continuity,
this means that there is $k$, $1\le k \le n-1$, such that
$\limsup_{q\to p, q\neq p} l_\Omega(q)=k$ and that $l_\Omega(p)\le k-1$.
Since we assume \eqref{est}, $k=2(m-1)$.

We refer the reader to \cite{Cat} for a complete definition of the Catlin
multitype of $\Omega$ at $p$.
Here we will only recall that it is an $n$-tuple $M=(1,m_2,\dots,m_n)$,
$2\le m_2 \le \dots \le m_n$,
with the following property: if $\Lambda=(\lambda_1,
\dots,\lambda_n)\in [1,\infty]^n$ is a weight  such that $\Lambda < M$
in the lexicographical order, then there exists another weight $\Lambda'$
with $\Lambda < \Lambda' \le M$ and a defining function $\rho$ for $\Omega$ and
a system of complex coordinates which is
distinguished with respect to $\Lambda'$, i.e.
$$
\mbox{if }\alpha, \beta \in \N^n \mbox{ verify }
\sum_{j=1}^n \frac{\alpha_j + \beta_j}{\lambda'_j} < 1,
\mbox{ then }
\frac{\partial^{|\alpha|+|\beta|}\rho}{\partial z^\alpha \partial\bar z^\beta}(p)=0.
$$
Catlin proved that the multitype is upper semicontinuous with respect to
lexicographical order \cite[Theorem 1 (1)]{Cat}.

It is well known that the multitype of a strictly pseudoconvex point
is $(1,2,\dots,2)$. If $l_\Omega \equiv k$ in a neighborhood of a point $q$,
we can take holomorphic coordinates $(z_1, z_2, \dots ,z_n)$ such that
$\{z_{k+2}=\cdots=z_n=0\}$ represents the $n-1-k$ complex dimensional manifold
contained in $\partial \Omega$ and passing through $q$ which exists by Sommer's
theorem \cite{Som}. Then all the derivatives of $\rho$ vanish in those directions,
while the complex tangential directions $z_2,\dots,z_{k+1}$ are ``strictly
pseudoconvex" directions and $z_1$ is the complex normal direction.
Therefore the type will be given by $m_j=2$, $2\le j \le k+1$,
and $m_j=\infty$, $k+2\le j \le n$.

Our assumption on $p$ and the upper semi continuity of the multitype now imply
that at $p$, $m_j\ge 2$, $2\le j \le k+1$,  $m_j=\infty$, $k+2\le j \le n$,
and $m_{k+1}>2$ (otherwise $l_\Omega(p)=k$).

\begin{lemma}
\label{rhoest}
Under the above assumption on the multitype, for any $r$ large enough
and any $\eps>0$,
there exists $A_r >0$, holomorphic coordinates $z_1,\dots,z_n$
and $U$ a neighborhood of $0$ such that for any $z\in U$,
\begin{multline*}
\rho(z) \le \Re z_1 + \eps |z_1| + A_r\left( \sum_{j=2}^k |z_j|^2 + |z_{k+1}|^{3} +
\sum_{j=k+2}^n |z_j|^r\right).
\end{multline*}
\end{lemma}
Accepting Lemma \ref{rhoest}, we see that for  $c>0$
an appropriate constant and $x>0$ small enough,
there is a polydisk contained in $\Omega$
centered at $p+xn_p=(-x,0,\dots,0)$
with respective radii $cx$ in the $z_1$ direction, $cx^{1/2}$ in the $z_2, \dots
z_k$ directions, $cx^{1/3}$ in the $z_{k+1}$ direction, and
$cx^{1/r}$ in the $z_{k+2}, \dots
z_n$ directions.
The usual volume estimate yields that
\begin{multline*}
\log k_\Omega (p+xn_p) \le \left( 1 + \frac{k-1}2 + \frac13 + \frac{n-1-k}r\right)\log \frac1x +O(1)\\
\le \left(\frac56+\frac{k}2+ \frac{n-1-k}r\right) \log \frac1x  +O(1) \ll\left(1+\frac{k}2\right)\log \frac1x
\end{multline*}
if $r$ is chosen large enough.  This contradicts  \eqref{est} with
$m=1+ \frac{k}2$, which was our assumption, so in fact $l_\Omega(p)=k$.
\end{proof}

\begin{proof*}{\it Proof of Lemma \ref{rhoest}.}
Pick an integer $r\ge 3$ such that $\frac1{m_{k+1}} + \frac1{2r} <\frac12$.

By the assumption on multitype, there exists an
admissible weight $\Lambda=(1,\lambda_2,\dots,\lambda_n)$
with $\lambda_j \ge m_j\ge 2$, $2\le j \le k$, $\lambda_{k+1}\ge m_{k+1}>2$,
$\lambda_j \ge r^2$, $k+2\le j \le n$.  There is an admissible system of coordinates
in which we can write the Taylor formula up to order $r$:
$$
\rho(z)= \Re z_1 + \sum_{2\le |\alpha|+|\beta| \le r} \frac1{\alpha ! \beta !}
\frac{\partial^{|\alpha|+|\beta|}\rho}{\partial z^\alpha \partial\bar z^\beta}(0)
z^\alpha \bar z^\beta + o(|z|^r).
$$
The remainder term will be dominated by the desired estimates. We now estimate the
terms in the sum. For any nonzero term, we must have
$\sum_{j=1}^n \frac{\alpha_j+\beta_j }{\lambda_j} \ge 1$.

If $\alpha_1+\beta_1 \ge 1$, then the corresponding term is an $O(|z_1||z|)$,
so will be bound by $\eps |z_1|$. From now on assume $\alpha_1+\beta_1 =0$.

If $\sum_{j=2}^k \alpha_j+\beta_j \ge 2$, then the corresponding term is an
$O(\sum_{j=2}^k |z_j|^2)$.  Write $\alpha'= (\alpha_1, \dots, \alpha_k)$,
$\beta'= (\beta_1, \dots, \beta_k)$, $\alpha''= (\alpha_{k+1}, \dots, \alpha_n)$,
$\beta''= (\beta_{k+1}, \dots, \beta_n)$.

If $\sum_{j=2}^k \alpha_j+\beta_j =1$, then
$$
\left| z^{\alpha'}\bar z^{\beta'}
z^{\alpha''}\bar z^{\beta''} \right|
= O\left( \left| z^{2\alpha'}\bar z^{2\beta'}\right|
+   \left| z^{2\alpha''}\bar z^{2\beta''}\right| \right),
$$
so the first term in that last sum is an $O(\sum_{j=2}^k |z_j|^2)$, and the second one
has exponents which verify
$$
\sum_{j=k+1}^n \frac{2\alpha_j+2\beta_j }{\lambda_j} \ge 2(1-\frac1{m_2})\ge 1,
$$
so will be treated as the next case.

If $\sum_{j=2}^k \alpha_j+\beta_j =0$, then either
$\alpha_{k+1} + \beta_{k+1}\ge 3$ and the corresponding term is an
$O(|z_{k+1}|^3)$. Otherwise,
$$
\frac1{r^2} \sum_{j=k+2}^n \alpha_j+\beta_j  \ge
\sum_{j=k+2}^n \frac{\alpha_j+\beta_j }{\lambda_j} \ge 1 - \frac2{\lambda_{k+1}} > \frac1r,
$$
so $\sum_{j=k+2}^n \alpha_j+\beta_j \ge r$ and the corresponding term is an
$O(\sum_{j=k+2}^k |z_j|^r)$.
\end{proof*}

\section{Proof of the estimates}
\label{pfest}

The key point will be the following:

\begin{lemma}
\label{bihol} Suppose that $\partial \Om$ is Levi-flat near
$p\in\partial\Om$. Then there exist neighborhoods $V\Subset U$ of
$p$ such that for any $q\in\partial\Om\cap V$ one may find a
biholomorphism $\Phi_q$ defined on $U$ such that $\Phi_q(q)=0$,
and $\Phi_q(\Om\cap U)=\{w\in U:\rho_q(w)<0\},$ where $\rho_q$ is
a $\mathcal C^2$-smooth function such that $\rho_q(w)=\Re w_1-f_q (w),$
$\mbox{ord}_0 f_q\ge 2$ and $f_q(0,w'')=0,$ where $w=(w_1,w'').$
Moreover, the Jacobian determinant of $\Phi_q$ is identically $1$,
the $C^2$-norms of $\Phi_q$ and $f_q$ are bounded on $U$,
uniformly in $q,$ and $||\Phi_q(z)+q-z||\le C||z''-q''||.$
\end{lemma}

Notice that this is not the same $\Phi$ as in section \ref{sufberg}.
The proof in this section may seem similar to that in section \ref{sufberg},
in that it also depends on a local foliation; the difference being that now
the leaves of the foliation have the maximum complex dimension $n-1$,
and so there is no vector in the complex tangent space to $\partial \Om$
that is transverse to that foliation (as was the case in section \ref{sufberg}).

\begin{proof} We may assume that $p=0$ and that a defining function
of $\Om$ near $0$ is given by
$$\rho(z) = \Re z_1-g(\Im z_1, z''),\quad\mbox{ord}_0 g\ge 2.$$

We know from \cite{BF} that $\partial \Om$ near 0 is foliated by complex manifolds of the form
$$
\left( \vphi (y,z''), z'' \right),\quad
(y,z'')\in(-\eps,\eps)\times\eps\D^{n-1},
$$
where $\eps>0,$ $\vphi$ is a $\mathcal C^2$-smooth function,
$\vphi(y,\cdot)\in\mathcal O(\eps\D^{n-1})$ and $ \vphi(y,0)=
g(y,0)+iy.$ The Implicit Function Theorem implies that there is a
$\mathcal C^2$-smooth function $y(q)$ such that $q =(\vphi(y(q),q''),q'').$

The following map will be the desired biholomorphism if we choose the neighborhoods in an
appropriate way:
$$\Phi_q (z_1,z'')=(z_1-\vphi(y(q),z''),z''-q'').$$
Indeed, let $\mathcal L_q$ denote the unique leaf in the foliation
passing through the point $q$. Then $\Phi_q(\mathcal L_q) = \{
w_1=0\},$ so $\{ w_1=0\} \subset \partial \Phi_q(\Om)$ (near $0$)
and we may set $\rho_q=\rho\circ\Phi_q^{-1}.$
\end{proof}

Now, we are ready to prove the estimates in Theorem \ref{koba}
under the respective conditions. It follows from the above lemma
that there exist neighborhoods $V\Subset U$ of $p$ and a constant
$\eps>0$ such that for any point $z\in D\cap V$ one has that
\begin{multline*}
\{w:\Re w_1+\eps|w_1|<0\}\cap(\eps^2\D^n)=:F\subset\Phi_{p(z)}(\Om\cap U)\subset\\
G:=\{w:\Re w_1-\eps|w_1|<0\}\cap(2\eps^2\D^n).
\end{multline*}
Note that if $w(z)=\Phi_{p(z)}(z),$ then
$||w(z)+(\delta_\Om(z),0'')||/\delta_\Om(z)\to 0$ as $z\to 0.$

Moreover, if $Y(z)=(\Phi_{p(z)})_{\ast,z}(X),$ then
$Y_1(z)=(1+O(||z||)X_N$ and $Y''(z)=X''.$ Using this and, for
example, the product property of the Kobayashi metric
($\kappa_{D_1\times D_2}=\max(\kappa_{D_1},\kappa_{D_2})$;
cf.~\cite{JP}) and a dilatation of the coordinates, one may find a
constant $c_2>0$ such that
\begin{multline*}
||X_N||/(c_2\delta_\Om(z))\le\kappa_G(w(z);Y(z))\le\kappa_{\Om\cap U}(z;X)\le\\
\kappa_F(w(z);Y(z))\le c_2||X_N||/\delta_\Om(z)+c_2||X||.
\end{multline*}
To get \eqref{growth}, it remains to use that $\kappa_\Om\le\kappa_{\Om\cap U}\le c_3\kappa_\Om$ (cf.~\cite{JP})
and the fact that $\kappa_\Om(z;X)\ge||X||/\mbox{diam}(\Om).$

The proof of the estimate \eqref{metric} is similar. Indeed, since
$m_{D_2}\le m_{D_1}$ and $k_{D_2}\le k_{D_1}$ if $D_1\subset D_2,$
then
$$\beta_G(w(z);Y(z)/s(w(z))\le\beta_{\Om\cap U}(z;X)\le\beta_F(w(z);Y(z))s(w(z)),$$
where $s=k_F/k_G.$ Using, for example, the product property of the
Bergman metric ($\beta^2_{D_1\times
D_2}=\beta^2_{D_1}+\beta^2_{D_2}$; cf.~\cite{JP}) and a dilatation
of the coordinates, one may find a constant $c_4>0$ such that
\begin{multline*}
||X_N||/c_2\delta_\Om(z)\le\beta_G(w(z);Y(z))\le\beta_F(w(z);Y(z))\le\\
c_4||X_N||/\delta_\Om(z)+c_4||Y(z)||,\quad s(w(z))\le c_4.
\end{multline*}
To complete the proof of the estimate \eqref{metric}, it remains
to use that $c_5\beta_\Om\le\beta_{\Om\cap U}\le\beta_\Om/c_5$
(see \cite[Proposition 1]{DFH}) and
$\beta_\Om(z;X)\ge||X||/\mbox{diam}(\Om).$

Finally, as explained at the beginning of section \ref{ranklevi}, the
estimate \eqref{kernel} follows from \eqref{berggrowth} in the
case where the rank $k=0$.  We could also use easier versions of
the arguments above.

\section{Sharp estimates}\label{sharp}

The estimate in Theorem \ref{koba} can be made sharp in the convex
case.

\begin{proposition}\label{convex} Let $\Om\subset\C^n$ be a domain with Levi-flat
boundary near a $\mathcal C^2$-smooth boundary point $p.$ Assume in addition that $\Om$ does not intersect
the real tangent hyperplane to $\partial\Om$ at any boundary point near $p$ (for example, if $\Om$ is convex).
Denote by $\alpha_\Om$ the Carath\'eodory or the Kobayashi metric.
Then there exists a neighborhood $U_p$ of $p$ and a constant $c_p>0$ such that
$$0\le\alpha_\Om(z;X)-||X_N||/2\delta_\Om(z)\le c_p||X||$$
for any $z\in\Om\cap U_p$ and any $X\in\C^n.$
\end{proposition}

Recall that the Carath\'eodory metric is defined by
$$\gamma_\Om(z;X)=\sup\{|f'(z)X|:f\in\mathcal O(D,\Bbb D)\}.$$

\begin{proof} For the lower bound observe that $\Om$ is on the one side $\Pi_{p_z}$
of the real tangent hyperplane plane to $\partial\Om$ at $p_z$ and hence
$$\gamma_\Om(z;X)\ge\gamma_{\Pi_{p_z}}(z;X)\ge||X_N||/2\delta_\Om(z).$$

For the upper estimate choose a neighborhood $V_p$ such that
$\Om_p=\Om\cap V_p$ is convex. Then
$\kappa_\Om\le\kappa_{\Om_p}=\gamma_{\Om_p}$ and hence
$$\kappa_\Om(z;X)\le\kappa_{\Om_p}(z;X_N)+\kappa_{\Om_p}(z;X_T).$$
We already know the the Levi flatness implies that
$$\limsup_{z\to p}\kappa_{\Om_p}(z;X_T)\le c'||X_T||.$$
On the other hand, the proof of Proposition \ref{planar} implies that
$$\limsup_{z\to p}(\kappa_{\Om_p}(z;X_N)-||X_N||/2\delta_\Om(z))\le c''||X_N||$$
which completes the proof.
\end{proof}

Finally, we present a sharp estimate for invariant metrics in the planar case.

\begin{proposition}\label{planar} Let $p$ be a $C^{1,1}$-smooth boundary point of
a planar domain $D.$ Then
\begin{equation}\label{ck}
\limsup_{D\ni z\to p}|\alpha_D(z)-1/2\delta_D(z)|<\infty,
\end{equation}
\begin{equation}\label{ber}
\limsup_{D\ni z\to p}|\beta_D(z)-1/\sqrt 2\delta_D(z)|<\infty,
\end{equation}
where $\alpha$ is the Carath\'eodory or Kobayashi metric, and $\beta_D$ is the Bergman metric,
all taken at the unit vector. Moreover,
\begin{equation}\label{ker}
\limsup_{z\to p}|k_D(z)-1/2\sqrt\pi\delta_D(z)|<\infty.
\end{equation}
\end{proposition}

In the $C^1$-smooth case a weaker result is known, namely (see \cite{JN}, Proposition 2 and the remark in the end)
$$\lim_{z\to p}\alpha_D(z)\delta_D(z)=1/2,\quad \lim_{z\to p}\beta_D(z)\delta_D(z)=1/\sqrt 2.$$

\begin{proof} Let $r$ be the double signed distance to $\partial D.$
For $\zeta\in D$ near p, set $\Phi_\zeta(z)=\partial r/\partial z(p_\zeta)(z-p_\zeta),$
$D_\zeta=\Phi_\zeta(D)$ and $\delta_\zeta=\Phi_\zeta(\zeta).$ Note that
$$\alpha_D(z)=\alpha_{D_\zeta}(\eta_\zeta),\quad\delta_D(z)=\eta_\zeta.$$
Then there exists an $\eps>0$ such that
$$\{|z-\eps|<\eps\}=:F_\eps\subset D_\zeta\subset G_\eps:=\{|z+\eps|>\eps\}$$
for any $\zeta\in D$ near $p.$ Hence
\begin{multline*}
\frac{1}{\eta_\zeta(2-\eta_\zeta/\eps)}=\kappa_{F_\eps}(\eta)\ge\kappa_{D_\zeta}(\eta_\zeta)\ge
\gamma_{D_\zeta}(\eta_\zeta)\ge\gamma_{G_\eps}(\eta_\zeta)=\frac{1}{\eta_\zeta(2+\eta_\zeta/\eps)}
\end{multline*}
which implies \eqref{ck}.

The proof of \eqref{ber} is similar. Let
$m_\Om(z)=\beta_\Om(z){k_\Om(z)}.$ Recall that $k_\Om\le
k_\Theta$ and $m_\Om\le m_\Theta$ if $\Theta\subset\Om.$ Then
$$\frac{\sqrt2(2+\eta_\zeta/\eps)}{\eta_\zeta(2-\eta_\zeta/\eps)^2}=\frac{m_{F_\eps}(\zeta)}{k_{G_\eps}(\zeta)}\ge
\beta_{D_\zeta}(\eta_\zeta)\ge\frac{m_{G_\eps}(\eta_\zeta)}{k_{F_\eps}(\eta_\zeta)}=\frac{\sqrt2(2-\eta_\zeta/\eps)}
{\eta_\zeta(2+\eta_\zeta/\eps)^2}$$ which implies \eqref{ber}.

We skip the proof of \eqref{ker}, since it is even easier than that of \eqref{ber}.
\end{proof}


\begin{thebibliography}{}

\bibitem{BF} D. E. Barrett, J. E. Fornaess,
{\it On the smoothness of Levi-foliations},
Publ. Mat. 32 (1988), 171--177.

\bibitem{Cat} D. Catlin, {\it Boundary invariants of pseudoconvex domains},
Ann. Math. 120 (1984), 529--586.

\bibitem{Cat2} D. Catlin, {\it Estimates of invariant metrics on pseudoconvex domains
of dimension two}, Math. Z. {\bf 200} (1989), 429--466.

\bibitem{Die} K. Diederich, {\it Das Randverhalten der Bergmanschen Kernfunktion
und Metric in streng pseudo-konvexen Gebieten}, Math. Ann. 187
(1970), 9--36.

\bibitem{DFH} K. Diederich, J. E. Fornaess, G. Herbort, {\it Boundary behavior
of the Bergman metric}, Proc. Symp. Pure Math. 41 (1984), 59--67.

\bibitem{Fu1} S. Fu, {\it Geometry of bounded domains and behavior of invariant
metrics}, Ph. D. thesis, Washington University in St. Louis, 1994.

\bibitem{Fu2} S. Fu, {\it Estimates of invariant metrics on pseudoconvex domains
near boundaries with constant Levi ranks}, arXiv:1203.1872 .

\bibitem{JN} M. Jarnicki, N. Nikolov, {\it Behavior of the Carath\'eodory metric
near strictly convex boundary points}, Univ. Iag. Acta Math. XL (2002), 7--12.

\bibitem{JP} M. Jarnicki, P. Pflug,
{\it Invariant distances and metrics in complex analysis}, de
Gruyter, Berlin-New York, 1993.

\bibitem{Hor} L. H\"ormander, {\it $L^2$ estimates and existence theorems for
$\bar\partial$ operator}, Acta Math. 113 (1965), 89-152.

\bibitem{Kra} S. Krantz, {\it The boundary behavior of the Kobayashi metric},
Rocky Mount. J. Math. 22 (1992), 227--233.

\bibitem{Krt} P. Kraut, {\it Zu einem Satz von F. Sommer \"{u}ber eine komplex-analytische
Bl\"{a}tterung reeller Hyperfl\"{a}chen im $C^n$}, Math. Ann. 174
(1967), 305--310.

\bibitem{Ohs} T. Ohsawa, {\it Reviews and questions on the Bergman kernel in
complex geometry}, preprint (2010).

\bibitem{OT} T. Ohsawa, K. Takegoshi, {\it On the extension of $L^2$ holomorphic
functions}, Math. Z. 195 (1987), 197--204.

\bibitem{Som} F. Sommer, {\it Komplex-analytische Bl\"atterung reeller Hyperfl\"achen
im $C^n$},  Math. Ann. 137 (1959), 392--411.

\end{thebibliography}
\end{document}